\documentclass[11pt]{amsart}

\usepackage[a4paper, width=17cm ]{geometry}
\usepackage{amssymb, amsthm, amsmath}
\usepackage{enumerate, amscd}
\newtheorem{theorem}{Theorem}[section]

\newtheorem{lemma}[theorem]{Lemma}
\newtheorem{corollary}[theorem]{Corollary}

\newtheorem{remark}[theorem]{Remark}

\newtheorem*{theorem*}{Theorem}{\bf}{\it}
\newtheorem*{proposition*}{Proposition}{\bf}{\it}
\newtheorem*{observation*}{Observation}{\bf}{\it}
\newtheorem*{lemma*}{Lemma}{\bf}{\it}

\theoremstyle{definition}
\newtheorem{definition}[theorem]{Definition}

\newtheorem{proposition}{Proposition}

\begin{document}

\title{On the Boundary Behavior of  Positive Solutions of Elliptic Differential Equations
%\thanks{
%This research is supported by the Chebyshev Laboratory  (Department of Mathematics and Mechanics, St. Petersburg %State University)  under RF Government grant 11.G34.31.0026.}
}

\author{Alexander Logunov
}

%\institute{A. Logunov \at 
%              St. Petersburg State University, Department of Mathematics and Mechanics,
%              14th Line 29B, Vasilyevsky Island, St.Petersburg 199178, Russia \\
  %            Tel.: (+7-812) 363-68-71\\
     %         Fax: (+7-812) 363-68-71\\
        %      \email{log239@yandex.ru}           
%}

\maketitle

\begin{abstract}

Let $u$ be a positive harmonic function in the unit ball $B_1 \subset \mathbb{R}^n$ and let $\mu$ be the boundary measure of $u$. Consider a point $x\in \partial B_1$ and let $n(x)$ denote the unit normal vector at $x$. Let $\alpha$ be a number in $(-1,n-1]$ and $A \in [0,+\infty) $. We prove that $u(x+n(x)t)t^{\alpha} \to A$ as $t \to +0$ if and only if $\frac{\mu({B_r(x)})}{r^{n-1}} r^{\alpha} \to C_\alpha A$ as $r\to+0$, where ${C_\alpha= \frac{\pi^{n/2}}{\Gamma(\frac{n-\alpha+1}{2})\Gamma(\frac{\alpha+1}{2})}}$. For $\alpha=0$  it follows from the theorems by Rudin and Loomis which claim 
 that a positive harmonic function has a limit along the normal iff the boundary measure has the derivative at the corresponding point of the boundary. For $\alpha=n-1$ it concerns about the point mass of $\mu$ at $x$ and it follows from the Beurling minimal principle.  For the general case of $\alpha \in (-1,n-1)$ we prove it with the help of the Wiener Tauberian theorem in a similar way to Rudin's approach. Unfortunately this approach works for a ball or a half-space only but not for a general kind of domain.  In dimension $2$ one can use conformal mappings and generalise the statement above to sufficiently smooth domains, in  dimension $n\geq 3$ we showed that this generalisation  is possible for $\alpha\in [0,n-1]$ due to harmonic measure estimates.  The last method leads to an extension of the theorems by Loomis, Ramey and Ullrich on non-tangential limits of harmonic functions to positive solutions of elliptic differential equations with Holder continuous coefficients.
%\begin{enumerate}[(i)]
%\item  The limit $ \lim\limits_{t\to+0} u(0,t)t^{\alpha}$  exists and equals $a$ for some constant $a\in [0,+\infty)$.
%\item  There exists $b\in[0,+\infty)$ such that $\frac{\mu({B(r)})}{r^{n-1}} r^{\alpha} \to b$ when $r \to +0$.
%\end{enumerate}
% If any of the limits above exists, then the limit values are related by
%$$a=b\frac{\Gamma(\frac{n-\alpha+1}{2})\Gamma(\frac{\alpha+1}{2})}{\pi^{n/2}}$$  
%Let $u$ be a positive solution of an elliptic differential equation in a domain with a sufficiently smooth boundary and %$\mu$ be the boundary measure of $u$. We investigate properties of $\mu$ ensuring the existence of a non-tangential %(or normal) limit of $u$ at a given boundary point. We generalize results by W.Rudin, L.Loomis, W.Ramey and D.Ullrich %on the boundary behavior of positive functions harmonic in a halfspace. We also extend (in the same spirit) the %Beurling minimal principle to some class of elliptic operators. 
%\keywords{Non-tangential limits \and Beurling minimal principle \and Green function \and  Harmonic measure \and %Boundary measure}

 %\subclass{35J67 \and 31B25  \and 35J08 \and 31B05 }
\end{abstract}
\section{Introduction}
\label{introduction}
Let $K(x,t):=\frac{ct}{(|x|^2+t^2)^{n/2}}$ be the Poisson kernel in the upper halfspace $\mathbb{R}^n_+:=\{(x,t): x \in \mathbb{R}^{n-1}, t>0 \}$, $c$ is equal to $\frac{\Gamma(n/2)}{\pi^{n/2}}$.
 For any positive and harmonic in $\mathbb{R}^n_+$ function $u$ there is the unique representation (see \cite{WR})
\begin{equation} \label{umu}  u(x,t)= Ct+ \int\limits_{\mathbb{R}^{n-1}}K(x-\xi,t)d\mu (\xi) 
\end{equation}  
for some constant $C\geq0$ and Borel measure $\mu$ (non-negative) on $\mathbb{R}^{n-1}$ such that 
\begin{equation}
 \label{mu}  \int\limits_{\mathbb{R}^{n-1}}\frac{1}{(|x|^2+1)^{n/2}}d\mu (x) < +\infty.
\end{equation} 
The measure $\mu$ is called the boundary measure of $u$. If $u$ is continuous up to the boundary of $\mathbb{R}^n_+$, then $\mu=u|_{_{\partial \mathbb{R}^n_+}} \cdot dS $, where $dS$ - is the Lebesgue measure on $\partial \mathbb{R}^n_+$. 
 If $C$ in the representation \eqref{umu} is equal to 0, we will say that $u=u_\mu$. 
We introduce the notion of the order of  growth.
\begin{definition} \label{order}
  We will say that $u$ has the order of growth $\alpha$ at the point $x \in  \partial \mathbb{R}^n_+$ if $\alpha=\inf \{\kappa: u(x+t\bar n(x))  t^{\kappa}{\to 0} \hbox{ as $t$ goes to $0$} \}$.
\end{definition}
\begin{remark}
  The order of growth of positive harmonic function can not be less than $-1$. If $\frac{u(x+t\bar n(x))}{t} \to 0$ as $t\to +0$, then $u\equiv 0$.
\end{remark}

\begin{remark}
 The order of growth of positive harmonic function can not be greater than $n-1$.
\end{remark}
  Two remarks above easily follow from the representation \eqref{umu} and the inequality \eqref{mu}.
%the first remark also follows from the classical Harnack inequality and the last remark also follows from the Beurling %minimal principle (see \cite{BD1},\cite{VM} or section \ref{tBMP}).

 The goal of this paper is to bind the boundary behavior of $u$ (like growth along the normal) and the properties of the boundary measure $\mu$.  
 The next theorem provides the criterion of the polynomial growth along the normal. 
\begin{theorem} \label{criterium}
 Let $\alpha \in (-1,n-1)$. The following statements are equivalent:
\begin{enumerate}[(i)]
\item  The limit $ \lim\limits_{t\to+0} u(0,t)t^{\alpha}$  exists and equals $a$ for some constant $a\in [0,+\infty)$.
\item  There exists $b\in[0,+\infty)$ such that $\frac{\mu({B(r)})}{r^{n-1}} r^{\alpha} \to b$ as $r \to +0$.
\end{enumerate}
 If any of the limits above exists, then the limit values are related by
$$a=b\frac{\Gamma(\frac{n-\alpha+1}{2})\Gamma(\frac{\alpha+1}{2})}{\pi^{n/2}}$$  
\end{theorem}
 Theorem \ref{criterium} implies the following remark.
\begin{remark}
Let $u$, $\mu$ be as above. The order of growth of $u$ at $O$ equals $\alpha$ $\Longleftrightarrow$ $\alpha= \inf \{\kappa:   \frac{\mu({B(r)})}{r^{n-1}} r^{\kappa}\to 0 \hbox{ as $r$ goes to $0$} \}$. 
\end{remark}

  Theorem \ref{criterium} is true for $\alpha=n-1$ as well, this case follows from the Beurling minimal principle, we will briefly discuss it in section \ref{tBMP}. For $\alpha=-1$ the theorem fails. The case $\alpha=0$ is related to existence of the limit along the normal, this case was treated in \cite{LL} (n=2) and \cite{WR} ($n\geq 3$). In the last paper the main tool was the Wiener Tauberian theorem, we will also use it in the proof of theorem \ref{criterium}. However this approach deeply relies on the special form of our domain (half-space) and there is an obstacle to go to any other kind of domain. In dimension $2$ one can use conformal mappings and easily generalise the theorem above to sufficiently smooth domains, in  dimension $n\geq 3$ we showed that this generalisation  is possible for $\alpha\in [0,n-1]$ due to harmonic measure estimates, we do not treat the case $\alpha \in (-1,0)$. In order to compensate the lack of conformal mappings in higher dimensions we consider elliptic operators of the second order. Sometimes we perform deformations of the coordinates to make a general domain flat near a fixed boundary point, doing this step we lose harmonicity of a considered function and work with positive solutions of elliptic differential equations with the help of the asymptotic estimates of the Green function.

 Given a point $P\in \partial \Omega$ we are also interested in conditions to be imposed on $\mu$ near $P$ to ensure the existence of a finite non-tangential limit of $u$ at $P$. For the case of a half-space and the Laplace operator Loomis  showed for $n=2$, and  Ramey, Ullrich showed for $n\geq 3$ that there is a criterion in terms of smoothness of the boundary measure (see \cite{WRDU} or section \ref{nontangential limit} for the precise statement). The method of \cite{WRDU}  used the geometrical property of a half-space, namely $\mathbb{R}^n_+$ is invariant under the homothety transformation. That method treats smoothness of the boundary measure as a weak convergence of its rescalings to the Lebesgue measure and interprets non-tangential limit of a harmonic function as a normal convergence  of its dilatations to a constant. The elegant  approach in \cite{WRDU} avoids the Wiener Tauberian theorem.  We extend the criterion of existence of the non-tangential limit to solutions of elliptic equations in smooth domains in theorem \ref{NT2}.  It would be interesting to find any non-tangential analogues of the theorem \ref{criterium} in higher dimensions,  we refer the reader to \cite{AK},\cite{GFW2} for this analogue in dimension $n=2$, and to its connection to the Lindelof theorem.
 
We denote by $\Omega$ a domain in $\mathbb{R}^n$ with a sufficiently smooth boundary and by $L$ an elliptic differential operator, $$L:= \sum\limits_{i,j=1}^{n}a_{ij}\frac{\partial^2}{\partial x_i\partial x_j}+ \sum\limits_{i=1}^{n}b_i\frac{\partial}{\partial x_i}+c,$$
 where $a_{ij}$, $b_i$ and $c$ are functions defined in $\overline{\Omega}$. Let $\alpha\in(0,1)$ and $\lambda\geq 1$.
\begin{definition} \label{L+} Denote by $\mathrm{L}^+(\lambda,\alpha,\Omega)$ the class of all operators $L$ enjoying the following properties:
\begin{enumerate}[(a)]
\item $\lambda |\xi|^2 \geq \sum\limits_{i,j=1}^{n} a_{ij}(x)\xi_i\xi_j\geq \lambda^{-1}|\xi|^2$ for any $x\in \overline\Omega$ and $\xi \in \mathbb{R}^n$.
\item $\sum\limits_{i,j=1}^{n}|a_{ij}(x)-a_{ij}(y)|+\sum\limits_{i=1}^{n}|b_i(x)-b_i(y)|+|c(x)-c(y)| \leq \lambda|x-y|^\alpha$ for all $x,y \in \overline\Omega$.
\item $\sum\limits_{i,j=1}^{n}|a_{ij}(x)|+\sum\limits_{i=1}^{n}|b_i(x)|+|c(x)| \leq \lambda$ for any $x \in \overline\Omega$.
\item $c(x)\leq0$ for any $x\in \overline\Omega$.
\end{enumerate}

\end{definition}
\begin{definition}
A function $u\in C^2(\Omega)$ is called $L$-harmonic if $Lu=0$ in $\Omega$.
\end{definition}
Any positive $L$-harmonic function $u$ admits a Poisson-like integral representation
\begin{equation} \label{representation 1} u(x)=\int\limits_{\partial\Omega}\frac{\partial G_L(\xi,x)}{\partial \nu(\xi)}d\mu(\xi), x\in{\Omega}
\end{equation}
 with a Borel measure $\mu=\mu_u$ on $\partial\Omega$ (we assume $\Omega$ is $C^{2,\varepsilon}$-smooth and bounded, $L \in \mathrm{L}^+(\lambda,\alpha,\Omega)$). We will call $\mu$ the boundary measure of $u$. We denote by $G_L$ the Green function of $L$ for the domain $\Omega$ (see \cite{AA}, \cite{HHSM}); $n(x)$ and $\nu(x):=\Big( a_{ij}(x) \Big) n(x)$ are normal and conormal vectors at the boundary point $x\in \partial\Omega$. We will implicitly use certain smoothness of $G_L(x,y)$ up to the boundary of $\Omega$ as a function of $x$  whenever $x\neq y$ (see \cite{HHSM2}, lemma 2.1).
Note that the representation formula (\ref{representation 1}) has a generalisation in terms of the Martin's boundary for more general domains (see \cite{AA}, theorem 6.3), in our particular case the Green function is sufficiently smooth up to the boundary and the Martin's representation implies \eqref{representation 1}.

 Our paper is organized as follows. In section \ref{GT} we prove the theorem \ref{criterium} with the help of  the Wiener Tauberian theorem, in section \ref{G-estimates} we reproduce some known information on  the $L$-harmonic measure (i.e. $\frac{\partial G_L(x,\xi)}{\partial\nu(\xi)}dS(\xi)$, $dS$ being the surface Lebesgue measure on $\partial \Omega$) and recall some pointwise estimates of the Green function $G_L$, section \ref{asymptotic formula} is devoted to asymptotic estimates of  $L$-harmonic measure.  This information is used to compensate the lack of conformal mappings in higher dimensions and to reduce our problems for positive $L$-harmonic functions to usual $\Delta$-harmonic functions, this reduction step is done and exploited in section \ref{applications}, where we extend theorem \ref{criterium} to smooth domains, prove the criterion of existence of non-tangential limit for $L$-harmonic functions and briefly discuss the Beurling minimal principle. The term "Beurling's minimal principle" was introduced in the paper \cite{VM}  where the result  of A.Beurling on the behavior of positve harmonic functions was extended to higher dimensions and generalized to positive solutions of elliptic operators in divergence form in sufficiently smooth domains. It can be viewed as the condition on growth of a positive harmonic function along the sequence of points which ensures the boundary measure to have a point mass.  One of the ideas used in \cite{VM} concerned the asymptotic estimates for the Green function near the boundary and it allowed to go from a halfspace and the Laplace operator to $C^{1,\varepsilon}$ -smooth domains and elliptic operators in the divergence form. We will use this idea to show that the Beurling minimal principle holds for some class of elliptic operators in non-divergence form as well.

 We conclude this introduction  with a reminder concerning the role played by positivity of $L$-harmonic functions.  The existence of a normal limit of the -Poisson
 integral of a (not necessarily positive) charge $\mu$ at $P\in\partial B$ is implied by the existence of its symmetric derivative at $P$, it is the classical fact due to Fatou. But this assertion cannot be reversed, see counterexample in \cite{LL}. However the converse of the Fatou teorem is true for the Poisson integrals of measures (i.e. non-negative charges). These results may be interpreted as tauberian theorems with the positivity as a tauberian condition (see section 12.12 of \cite{H} for some results of the same kind based on the Wiener tauberian theorem and, in particular, implying (with some effort) an analogous tauberian theorem for solutions of the heat equation). There exist other tauberian conditions for which the converse of the Fatou theorem holds, see \cite{WRDU},\cite{CD}, \cite{ED2}, \cite{BS}.
\section{Proof of Theorem \ref{criterium}} \label{GT}
  Consider the multiplicative group $G$ on the set $\mathbb{R}_+$. The convolution of functions on $G$ is defined by 
\begin{equation}
\label{convolution}
(f\star g) (t) = \int\limits_{0}^{+\infty} f(t/s)g(s) d\ln(s)
\end{equation}
 
 Denote by $\{F_\alpha\}_{\alpha \in \mathbb{R}}$ the family of functions on $G$:
\begin{equation}
\label{characters}
F_{\alpha}(t):= t^{\alpha}
\end{equation}
 Let $dt$ denote the Lebesgue measure on $\mathbb{R}_+$.
  It's easy to see that for any function $g\in L^1(\mathbb{R}_+, \frac{dt }{t^{\alpha+1}} )$ 
\begin{equation}
 F_\alpha \star g = F_\alpha \cdot \int\limits_{0}^{+\infty} g(t)  \frac{dt }{t^{\alpha+1}} 
\end{equation}
  We recall some useful convergence claims concerning the convolution on $G$.
\begin{proposition} \label{convolution}
 Suppose that functions $f$, $g$ satisfy: $\frac{f}{F_\alpha} \in L^{\infty}(\mathbb{R}_+, {dt})$ and  $\frac{g}{F_\alpha}\in L^1(\mathbb{R}_+, d \ln(t)  )$. Then the following properties hold:
\begin{enumerate}[(i)] 
\item \label{convolution identity}$\frac{f}{F_\alpha}\star \frac{g}{F_\alpha}=\frac{f\star g}{F_\alpha}$
\item $\sup\limits_{t \in \mathbb{R_+}}|\frac{(f\star g)(t)}{t^\alpha}|\leq \| \frac{f}{F_\alpha} \|_{L^\infty(\mathbb{R}_+, {dt})} \cdot \|\frac{g}{F_\alpha}\|_{L^1(\mathbb{R}_+, d\ln(t) )}$
\item \label{convolutioniii} If $\frac{f(t)}{t^\alpha} \mathop{\to}\limits_{t\to 0} a$  for some $a\in \mathbb{R}$, then $$\frac{(f\star g)(t)}{t^\alpha} \mathop{\to}\limits_{t\to 0} a \int\limits_{0}^{+\infty} g(t)  \frac{dt }{t^{\alpha+1}}  $$

\end{enumerate}
\end{proposition}
 Next, we will formulate the Wiener Tauberian theorem (see \cite{H}).
\begin{theorem} \label{Wiener}Suppose that Fourier transform(on $G$) of the function $f\in L^1(\mathbb{R}_+,d\ln(t))$ has no zeroes, i.e. $\hat f(y):=\int\limits_{\mathbb{R}_+} f(t)t^{-iy} d \ln(t)$ does not vanish for any $y\in \mathbb{R}$. Then the following holds:
\begin{enumerate}
\item ${Lin}(\{f(\lambda t)\}_{t \in \mathbb{R}})$  is dense in $ L^1(\mathbb{R}_+,d\ln(t)) $.
\item If $(f\star g)(t) \mathop{\to}\limits_{t\to 0} a$ for some $a \in \mathbb{R}$ and $g \in \L^{\infty}(\mathbb{R}_+, d\ln(t))$, then for any $h\in L^1(\mathbb{R}_+,d\ln(t))$  $$(h\star g)(t) \mathop{\to}\limits_{t\to 0} a \cdot \frac{\hat h(0)}{\hat f(0)}$$

\end{enumerate}
\end{theorem}
 The next corollary is a straightforward consequence of the theorem \ref{Wiener} and the identity (\ref{convolution identity}) from proposition \ref{convolution}.
\begin{corollary}\label{corollary Wiener}
Let $\alpha$ be a fixed number from $\mathbb{R}$. Suppose that function $f$ satisfies
\begin{enumerate}[(a)]
\item $\frac{f}{F_\alpha} \in L^1(\mathbb{R}_+,d\ln(t)) $
\item The Fourier transform of $\frac{f}{F_\alpha}$ has no zeroes.
\end{enumerate}
 If $\frac{f\star g}{F_\alpha}(t) \mathop{\to}\limits_{t\to 0} a$ for some $a \in \mathbb{R}$ and $g$ such that $\frac{g}{F_\alpha} \in \L^{\infty}(\mathbb{R}_+, d\ln(t))$, then for any $h\in L^1(\mathbb{R}_+,\frac{d\ln(t)}{t^\alpha})$  $$\frac{h\star g}{F_\alpha}(t) \mathop{\to}\limits_{t\to 0} a \cdot \frac{\int\limits^{\mathbb{R}_+} h(t)\frac{d\ln(t)}{t^\alpha}}{\int\limits_{\mathbb{R}_+} f(t)\frac{d\ln(t)}{t^\alpha}}$$
\end{corollary}
 Now, we are allmost ready to prove the theorem \ref{criterium}.

 Without loss of generality we may assume that $C$ in the representation (\ref{umu}) of $u$ is equal to zero and $\mu$ is a finite measure with support contained in the unit ball $B_1$. Also we will assume that $\mu$ has no point mass at $O$(otherwise, it's clear that neither (i) nor (ii) in theorem \ref{criterium} happens). The representation (\ref{umu}) implies
$u(0,t)= \int\limits_{0}^{+\infty}\frac{ct}{(r^2+t^2)^{\frac{n}{2}}}d \mu(B(r)) $, where $c$ is equal to $\frac{\Gamma(n/2)}{\pi^{n/2}}$.  Integrating by parts we obtain \begin{equation} \label{representation 2}u(0,t)= \int\limits_{0}^{+\infty}\frac{ncrt}{(r^2+t^2)^{\frac{n}{2}+1}}\mu(B(r))dr \end{equation}
 We can rewrite it in the following form:
\begin{equation}\label{u=k*m} u(0,t)= (k\star M)(t)
\end{equation}
 Where $M(r):=  \frac{\mu(B(r))}{r^{n-1}}$ and $k(t):=\frac{nct}{(1+t^2)^{\frac{n}{2}+1}}$.  Note that $M(r)\leq \frac{K}{r^{n-1}}$, where $K$ is equal to the full variation of $\mu$.
 
We start with the proof of the implication 
{\bf (ii) $\Longrightarrow $(i)         }  in theorem \ref{criterium}.
 Using the condition (ii) and the inequality $M(t)\leq \frac{K}{t^{n-1}}$ we obtain $M(t)t^{\alpha} \in L^{\infty}(\mathbb{R}_+, d\ln(t))$. Since $k(t)t^{\alpha} \in L^1(\mathbb{R}_+,d\ln(t))$, we are able to apply  proposition \ref{convolution}(using the identity (\ref{u=k*m}))and obtain (i). Also it implies that constants in (i) and (ii) are related by $$a=b\cdot\int\limits_{0}^{+\infty} k(t)t^\alpha d \ln(t)=b\cdot \frac{\Gamma(\frac{n-\alpha+1}{2})\Gamma(\frac{\alpha+1}{2})}{\pi^{n/2}} $$
We will prove the last equality giving the explicit value of the integral with $k$  later when the Fourier transform of $k(t)t^{\alpha}$ will be calculated.

{\bf (ii) $\Longleftarrow $(i)         } 

 First, we note that $u(0,t)\geq  \int\limits_{t}^{2t}\frac{ncrt}{(r^2+t^2)^{\frac{n}{2}+1}}\mu(B(r))dr \geq K_1 \frac{\mu(B(t))}{t^{n-1}}=K_1 M(t)$, where $K_1$ is some positive constant depending only on the dimension $n$.  Hence $ \limsup\limits_{t\to+0} M(t)t^{\alpha}\leq \frac{1}{K_1}\lim u(0,t)t^\alpha$. Using the last observation and the inequality $M(t)\leq \frac{K}{t^{n-1}}$ we obtain $M(t)t^{\alpha} \in L^{\infty}(\mathbb{R}_+, d\ln(t))$. It has allready been mentioned that $k(t)t^{\alpha} \in L^1(\mathbb{R}_+,d\ln(t))$, and all that remains to apply the corollary of Wiener theorem is to compute the Fourier transform of $k(t) t^{\alpha}$. $$\widehat{kF_\alpha}(y)= \int\limits_{0}^{+\infty}\frac{nct}{(1+t^2)^{\frac{n}{2}+1}}t^{\alpha}{t^{-iy}}\frac{dt}{t}\mathop{=\!=\!=\!=}\limits^{s =\frac{1}{1+t^2}} 
\int\limits_{0}^{1}\frac{nc}{2}s^{\frac{n}{2}-1}\Bigl(\frac{1-s}{s}\Bigr)^{\alpha/2-iy/2-1/2}ds=$$ $$=\frac{nc}{2}B(\frac{n-\alpha+iy+1}{2},\frac{\alpha-iy+1}{2})=\frac{n}{2}\frac{\Gamma(n/2)}{\pi^{n/2}}B(\frac{n-\alpha+iy+1}{2},\frac{\alpha-iy+1}{2})=$$ $$=\frac{\Gamma(\frac{n-\alpha+iy+1}{2})\Gamma(\frac{\alpha-iy+1}{2})}{\pi^{n/2}}. $$
 This has no zeros. Corollary \ref{corollary Wiener} implies that 
\begin{equation} \label{weak convergence} t^\alpha(M\star g)(t)  \mathop{\to}\limits_{t\to +0} a 
\frac{\int\limits^{\mathbb{R}_+} g(t)t^\alpha d\ln(t)}{\int\limits_{\mathbb{R}_+} k(t)t^\alpha d\ln(t)}
\end{equation} for any $g\in L^1(\mathbb{R}_+,t^\alpha d\ln(t))$. 
 The last step is to deduce the strong convergence from the weak convergence (\ref{weak convergence}). First, we note that due to monotonicity of $\mu(B(r))$ the following inequality holds:
\begin{equation} \label{monotonicity}M(t_1)t_1^{n-1}\leq M(t_2)t_2^{n-1} \hbox{ for any } 0<t_1\leq t_2
\end{equation}
  Fix $\varepsilon>0$. Consider a pair of functions 
$$g_{\varepsilon}(t):= 
\left\{ \begin{matrix} \varepsilon^{-1} , && t \in[1,1+\varepsilon]\\
 0 ,&&  t \notin[1,1+\varepsilon]
\end{matrix} \right. $$
$$g_{-\varepsilon}(t):= 
\left\{ \begin{matrix} \varepsilon^{-1} , && t \in[1-\varepsilon,1]\\
 0 ,&&  t \notin[1-\varepsilon,1].
\end{matrix} \right. $$
 The inequality (\ref{monotonicity}) yields $\frac{(M\star g_{\varepsilon})(t)}{(1+\varepsilon)^{n-1}} \leq M(t) \leq \frac{(M\star g_{-\varepsilon})(t)}{(1-\varepsilon)^{n-1}}$.
 In a view of (\ref{weak convergence}) we obtain 
 $$\frac{a}{(1+\varepsilon)^{n-1}}\frac{\int\limits^{\mathbb{R}_+} g_{\varepsilon}(t)t^\alpha d\ln(t)}{\int\limits_{\mathbb{R}_+} k(t)t^\alpha d\ln(t)}\leq \liminf\limits_{t\to +0} M(t)F_\alpha(t) \leq$$ 
$$ \leq \limsup\limits_{t\to +0} M(t)F_\alpha(t)\leq \frac{a}{(1-\varepsilon)^{n-1}}\frac{\int\limits^{\mathbb{R}_+} g_{-\varepsilon}(t)t^\alpha d\ln(t)}{\int\limits_{\mathbb{R}_+} k(t)t^\alpha d\ln(t)}.$$
 Letting $\varepsilon\to 0$ we earn (ii) because $\int\limits^{\mathbb{R}_+} g_{\pm\varepsilon}(t)t^\alpha d\ln(t) \to 1$.  The proof of theorem \ref{criterium} is finished.
\section{Pointwise Estimates of Green function} \label{G-estimates}
 There is a wide literature concerning the estimates of Green function. For instance, see \cite{KOW},\cite{MGKOW}, \cite{HHSM}, \cite{AILR}, \cite{VM}, \cite{JS}, \cite{ZZ}. We will recall several pointwise estimates of Green function. 
\begin{theorem}[{\cite{HHSM}, \cite{KOW}, \cite{ZZ}}]
 \label{G1}
 Let  $\Omega$ be a $C^{1,1}$-smooth bounded domain in $\mathbb{R}^n$, $n\geq 3$ and $G(x,y)$ be the
 Green function for the Laplace operator $\Delta$. Then there is a positive constant K depending only on the diameter of $\Omega$, the curvature of $\partial \Omega$ (the Lipschitz constant in the definition of a $C^{1,1}$ domain)
%$(the radius $r$ such that for every point $x \in \partial \Omega$ there two inscribed circles )
 and on $n$ such that the following estimates hold
\begin{enumerate}[(i)]
\item $G(x,y)\leq K \min(1,\frac{d(x,\partial \Omega)}{|x-y|})\min(1,\frac{d(y,\partial \Omega)}{|x-y|})|x-y|^{2-n}$ 
\item $G(x,y)\geq \frac{1}{K} \min(1,\frac{d(x,\partial \Omega)}{|x-y|})\min(1,\frac{d(y,\partial \Omega)}{|x-y|})|x-y|^{2-n}$ 
\end{enumerate}
for any $x$, $y$ $\in \overline{\Omega}$.

\end{theorem}
 The theorem above is valid if we replace the Laplace operator by some other operator $L$  from the class $L^+(\lambda,\alpha,\Omega)$. It immediately follows from the next theorem.
 
\begin{theorem}[{\cite{HHSM}}]
\label{G2}
 Let  $\Omega$ be as above and suppose  that $L\in L^+(\lambda,\alpha,\Omega)$. Then there is a positive constant $K$ depending only on the diameter of $\Omega$, the curvature of $\partial \Omega$, on $\alpha$, $\lambda$ and  $n$, such that the following estimates for $G_L$, the Green function for $L$,  hold for any $x, y \in \overline{\Omega}$

\begin{equation} K^{-1} G_{\triangle}(x,y) \leq G_L(x,y)\leq K G_{\triangle}(x,y).
\end{equation}

\end{theorem}
         We need a tool to compare Green functions for different elliptic operators with close coefficients.  The tool is  the following theorem by H.Hueber and M.Sieveking. 
\begin{theorem}[{\cite{HHSM2}}]
\label{G3}
Assume that $\Omega$ is $C^{2,\alpha}$-smooth bounded domain and let  $\{L_n\}_{n=1}^{\infty}$ be a sequence of elliptic operators from $\mathrm{L}^+(\lambda,\alpha,\Omega)$ such that the coefficients of $L_n$  converge to the respective coefficients of $L\in \mathrm{L}^+(\lambda,\alpha,\Omega)$ uniformly in $\overline\Omega$. Then there is a sequence of numbers $K_n \geq 1$, such that  $K_n \to 1$ and

 $$K_n^{-1} G_{L_n}(x,y) \leq G_L(x,y)\leq K_n G_{L_n}(x,y), n=1,2,\dots, x,y \in \overline{\Omega}.$$
 In other words, $G_{L_n}(x,y)= G_L(x,y)(1+o(1))$ where $o(1)$ is uniform with respect to $x,y \in \overline\Omega$.

\end{theorem}

The next corollary is a simple consequence of theorems \ref{G1}, \ref{G2}. It provides two-sided estimates for the $L$-Poisson kernel.
\begin{corollary}[\cite{AILR}, \cite{JS},\cite{KOW}]
\label{P1}
 Let  $\Omega$, $L$, $G_L$ be as in theorems \ref{G1}, \ref{G2}  and $\nu(x)$ denote inner  conormal vector at  $x\in \partial \Omega$ with respect to $L$. Then there is a positive constant K such that the following inequalities are true for any $x\in \partial \Omega  $, $y \in \Omega$
$$ \frac{1}{K}\frac{ d(y,\partial\Omega)}{ |x-y|^n}\leq \frac{\partial G_L(x,y)}{\partial \nu(x)} \leq K \frac{ d(y,\partial\Omega)}{ |x-y|^n}. $$ 
\end{corollary}
  We denote the Lebesgue surface measure on $\partial \Omega$ by $dS$ and the $L$-harmonic measure viewed from point $y\in \Omega$ by $d\omega_y$. In sufficiently smooth domains the L-harmonic measure is absolutely continuous  with respect to the surface measure and we denote by $P_L(x,y,\Omega)$ the density of $\omega_y$ with respect to $dS$ at $x \in \partial \Omega$ (see (\ref{representation 1}) ). Sometimes we omit the indices $\Omega$ or $L$ and simply write $P(x,y)$ or $P_L(x,y)$. Note that $ \frac{\partial G_L(x,y)}{\partial \nu(x)}= \kappa P_L(x,y)$ where $\kappa$ is a constant depending only on the normalization of Green function. So corollary \ref{P1} is equivalent to 
\begin{equation} \label{pe}
 \frac{1}{K}\frac{ d(y,\partial\Omega)}{ |x-y|^n}\leq P_L(x,y) \leq K \frac{ d(y,\partial\Omega)}{ |x-y|^n}.
\end{equation} The next corollary follows immediately from theorem \ref{G3} and corollary \ref{P1}.
 \begin{corollary}
\label{P2}
 Assume $\Omega$ is a $C^{2,\alpha}$-smooth bounded domain and let  $\{L_n\}_{n=1}^{\infty}$ be a sequence of elliptic operators in $\mathrm{L}^+(\lambda,\alpha,\Omega)$ such that the coefficients of $L_n$  tend to the respective coefficients of $L\in\mathrm{L}^+(\lambda,\alpha,\Omega)$ uniformly in $\overline\Omega$. Then
$$P_{L_n}(x,y)= P_{L}(x,y)(1+o(1))$$ 
where $o(1)$ is uniform with respect to $x\in \partial \Omega$, $y \in \Omega$. 
\end{corollary}
 
\section{Asymptotic behavior of $L$-harmonic measure} \label{asymptotic formula}
 Throughout this section $\Omega$ will be a $C^{2,\alpha}$-smooth and bounded domain in $\mathbb{R}^n$, $n\geq 3$ and operator  $L$ is in $\mathrm{L}^+(\lambda,\alpha,\Omega)$.
 The main point of this section is that the asymptotic behavior of  $P_L(x,y)$ as $d(y,\partial\Omega) \to +0$ is similar to the behavior of harmonic measure in a halfspace.

\begin{theorem}
\label{P3}
Assume that the origin point $O$ is in $ \partial \Omega$ and $\big{(}a_{ij}(O) \big{)}_{i,j=1}^{n}$ is the identity matrix. Put $\kappa_n:=\frac{\Gamma (n/2)}{\pi^{n/2}}$. Then the following asymptotic identities are true:
\begin{enumerate}[(i)]
 \item \label{P3i}$$P_L(x,y)  \sim \kappa_n\frac{ d(y,\partial\Omega)}{ d(x, y)^n}  $$
as $x,y\to O, y \in \Omega, x\in\partial \Omega$.
\item \label{P3ii}$$P_L(x,y) = \kappa_n\frac{ d(y,\partial\Omega)}{ d(x, y)^n}(1+o(1)) +o(1)  $$
as $y\to O, y \in \Omega $, both $o(1)$ being uniform with respect to $x\in\partial\Omega$.
\item  \label{P3iii}$$\left \| P_L(\cdot,y) - \kappa_n\frac{ d(y,\partial\Omega)}{ d(\cdot, y)^n}  \right \|_{L^1(\partial\Omega,dS)}  \xrightarrow{
\begin{matrix}
y \in \Omega\\
y\to O
\end{matrix}} 0 .$$
\end{enumerate}

\end{theorem}
 The previous theorem follows from methods of \cite{JS} and \cite{VM} where similar estimates appear. We outline the proof based on the theorem on convergence of Green functions by Hueber and Sieveking  for reader's convenience.  First, we need the following lemma.
\begin{lemma}
\label{simlemma}
 Suppose that bounded $C^{2,\alpha}$-smooth domain $\Omega_1\subset\Omega$ has a common boundary part with $\Omega$ containing the origin point $O$: $B_r(O)\cap  \Omega_1= B_r(O)\cap  \Omega$ for some $r>0$, where $B_r(O)$ denotes the ball centered at $O$ and of radius $r$. Then 
$$P_L(x,y,\Omega)  \sim P_L(x,y,\Omega_1) $$
as $x,y\to O, y \in \Omega, x\in\partial \Omega$.
\end{lemma} 
%\begin{proof}
\begin{proof}[Proof of lemma \ref{simlemma}]
  Denote by $G_L(x,y,\Omega)$ the Green function in $\Omega$ for $L$ and by $G_L(x,y,\Omega_1)$ the Green function in $\Omega_1$ for $L$. Consider $H(x,y):= G_L(x,y,\Omega) - G_L(x,y,\Omega_1) $. For any $x \in \overline{\Omega}_1$ the function $H(x,y):= G_L(x,y,\Omega) - G_L(x,y,\Omega_1) $ is $L$-harmonic and continuous in $\overline{\Omega}_1$ as a function of $y$. Hence 
\begin{align} 
H(x,y)= \int\limits_{\partial \Omega_1} H(x,\xi) P_L(\xi, y,\Omega_1) dS(\xi)= \nonumber \\
= \int\limits_{\partial \Omega_1}\left(G_L(x,\xi,\Omega) - G_L(x,\xi,\Omega_1)\right)P_L(\xi, y,\Omega_1) dS(\xi)= \nonumber \\
=\int\limits_{\partial \Omega_1 \backslash B_r(O) }G_L(x,\xi,\Omega) P_L(\xi, y,\Omega_1) dS(\xi). \nonumber
\end{align}
 Let $x,y$ be in $ B_{r/2}(O)$, then by (i) from theorem \ref{G1} and by (\ref{pe}) we obtain
\begin{align}
 \int\limits_{\partial \Omega_1 \backslash B_r(O) }G_L(x,\xi,\Omega) P_L(\xi, y,\Omega_1) dS(\xi) \leq
 \int\limits_{\partial \Omega_1 \backslash B_r(O) } K_1 \frac{d(x,\Omega)}{|x-\xi|^{1-n}} P_L(\xi, y,\Omega_1) dS(\xi) \leq \nonumber 
\end{align}
\begin{align}
\leq  \int\limits_{\partial \Omega_1 \backslash B_r(O) } K_1 K_2 \frac{d(x,\partial\Omega) d(y,\partial \Omega_1)}{|x-\xi|^{1-2n}} dS(\xi) \leq \nonumber \\ \leq K_1K_2 \frac{2^{2n-1}}{r^{2n-1}} S(\partial \Omega_1)d(x,\partial\Omega) d(y,\partial \Omega_1) = K_3d(x,\partial\Omega) d(y,\partial \Omega_1) \nonumber
\end{align}
 for some positive constants $K_1,K_2,K_3$.
Thus $$H(x,y)\leq K_3 d(x,\partial\Omega) d(y,\partial \Omega_1).$$
Hence  for any $x\in \partial \Omega \cap B_{r/2}(O)$, $y \in \Omega_1\cap B_{r/2}(O)$ the following inequality holds
\begin{equation} \label{en1} \frac{\partial H(x,y)}{\partial\nu(x)}\leq K_4  d(y,\partial\Omega_1). \end{equation}
  According to (\ref{pe}) 
\begin{equation} \label{en2} P_L(x,y,\Omega_1)\geq \frac{1}{K_5}\frac{d(y,\partial\Omega_1)}{|x-y|^n}. \end{equation}
 Inequalities (\ref{en1}),(\ref{en2}) imply
 $$\frac{\partial H(x,y)}{\partial\nu(x)} = o(P_L(x,y,\Omega_1))$$
as $x,y\to O, y \in \Omega, x\in\partial \Omega$.
 Thus 
$$P_L(x,y,\Omega)=P_L(x,y,\Omega_1)+ \frac{\partial H(x,y)}{\partial\nu(x)}=(1+o(1))P_L(x,y,\Omega_1).$$
\end{proof}
\begin{remark}
 The statement of lemma \ref{simlemma} remains true if $L$ is the Laplace operator $\Delta$ and $\Omega$ is $\mathbb{R}^n_+$. The proof above still works using the explicit formula for $G_\Delta(x,y,\mathbb{R}^n_+)$. 
\end{remark}
\begin{proof}[Proof of theorem \ref{P3}.]   Proposition (\ref{P3ii}) follows from (\ref{P3i}) and \eqref{pe}, (\ref{P3iii}) follows from (\ref{P3ii}) and the inequality $\left \| P_L(\cdot,y)  \right \|_{L^1(\partial\Omega,dS)}\leq 1$. All that remains to prove is (\ref{P3i}). 

 We may assume $\partial\Omega$ is flat in a neighborhood of  $O$ using appropriate smooth change of  coordinates   $T:\mathbb{R}^n \to \mathbb{R}^n$ preserving Holder-continuity of the 
cofficients of $L$ and conditions (a)-(d) in definition \ref{L+} with a slight perturbation of the constants $\alpha$, $\lambda$ (see \cite{HHSM2} for a careful treatment of the transform of properly normalized Green function under a change of the coordinates). The Jacoby matrix of $T$ at $O$ has to be orthogonal to preserve $A(O)$ as the identity matrix.

 Consider a convex and $C^{\infty}$-smooth domain $Q\subset \Omega$ with a flat boundary part in common with $\partial\Omega$: $B_r(O)\cap Q = B_r(O)\cap \Omega$ for some $r>0$. Consider the sequence of operators $$L_k:= \sum\limits_{i,j=1}^{n}a_{ij}(\frac{1}{k}x)\frac{\partial^2}{\partial x_i \partial x_j} + \sum\limits_{i=1}^{n}\frac{1}{k}b_i(\frac{1}{k}x) \frac{\partial}{\partial x_i}+\frac{1}{k^2}c(\frac{1}{k}x) \hbox{, } k=1,2 \dots\infty.$$
We write $X\mathop{\sim}\limits^{1+\varepsilon}Y$ whenever $\frac{X}{1+\varepsilon}$ is asymptotically less than $Y$ and $Y$ is asymptotically less than $X(1+\varepsilon)$.
  Applying corollary \ref{P2} in $Q$ to this sequence we obtain for any $\varepsilon>0$ a  $k=:k(\varepsilon)$ such that   $P_{L_k}(x,y,Q)\mathop{\sim}\limits^{1+\varepsilon} P_{\triangle}(x,y, Q)$ as  $x,y\to O, y \in Q, x\in\partial Q$.  Hence $$P_{L}(x,y,\frac{1}{k}Q)\mathop{\sim}\limits^{1+\varepsilon} P_{\triangle}(x,y,\frac{1}{k}Q).$$

 According to lemma \ref{simlemma}  $P_{L}(x,y,\frac{1}{k}Q)\sim  P_{L}(x,y,\Omega)$ and $P_{\triangle}(x,y,\frac{1}{k}Q)\sim  P_{\triangle}(x,y,\mathbb{R}^n_{+})$. Applying the explicit formula $$P_{\triangle}(x,y,\mathbb{R}^n_{+})=\kappa_n\frac{ d(y,\partial\Omega)}{ d(x, y)^n} $$ we obtain $P_L(x,y)  \mathop{\sim}\limits^{1+\varepsilon} \kappa_n\frac{ d(y,\partial\Omega)}{ d(x, y)^n}  $. Letting $\varepsilon \to 0$ we finally get (\ref{P3i}).
\end{proof}

\section{Applications of the asymptotic formula for $L$-harmonic measure} \label{applications}

 Two theorems are stated below. The first one establishes the relation between the boundary behavior of positive solutions of two different elliptic equations with the same boundary measures. The second links nontangential limits in a halfspace and in a smooth domain. We will use these theorems later in sections \ref{nontangential limit}, \ref{tBMP}.
\begin{theorem} \label{sim1}
 Suppose that two functions $u$ and $\tilde u$ in a $C^{2,\alpha}$-smooth and bounded domain $\Omega$ enjoy the following properties:
\begin{enumerate} 
\item $u$ is positive and harmonic in $\Omega$.
\item  $\tilde u$ is  positive and $L$-harmonic in $\Omega$, where $L\in L^+(\lambda,\alpha,\Omega)$. 
\item  The boundary measure $\mu$ of $u$ (see section \ref{introduction}) coincides with the boundary measure of $\tilde u$, that is $\tilde u(\cdot)=\int\limits_{\partial\Omega}P_L(x,\cdot)d\mu(x)$ and $u(\cdot)=\int\limits_{\partial\Omega}P_\Delta(x,\cdot)d\mu(x)$.
\item  The point $O\in \partial \Omega$ and the matrix $A(O):=\big{(}a_{ij}(O) \big{)}_{i,j=1}^{n}$ is the identity matrix.
\end{enumerate}
 Then 
\begin{enumerate}[(i)]
\item For any $\varepsilon>0$ $\exists \delta>0$, such that $$u(y)(1-\varepsilon)-\varepsilon \leq\tilde u(y)\leq u(y)(1+\varepsilon)+\varepsilon$$ for any $y\in\Omega$: $d(y,O)<\delta$.
\item If  $u\geq 1$ in $\Omega$, then 
 $$\lim\limits_{y\to O; y\in\Omega} \frac{u(y)}{\tilde u(y)}=1. $$
\end{enumerate}

\end{theorem}
 We will prove theorem \ref{sim1} a few moments later.
\begin{theorem} \label{sim2}
 Assume the boundary of $C^{2,\alpha}$-smooth and bounded domain $\Omega$ contains the origin $O$, $\mathbb{R}^n_0$ being the tangent plane to $\partial \Omega$ at $O$ and operator $L\in L^+(\lambda,\alpha,\Omega)$. 
Let $\tilde u$ be a positive $L$-harmonic function in  $\Omega$ and $\tilde\mu$ be its boundary measure. Consider the orthogonal projection $Pr: \partial \Omega\cap B_\varepsilon(O)\to \mathbb{R}^n_0$, where $\varepsilon>0$ is sufficiently small. Define the finite Borel measure $\mu$ on $\mathbb{R}^n_0$ by $\mu(E)=\tilde \mu(Pr^{-1}(E))$. Let $u$ be the harmonic continuation of $\mu$ into the halfspace containing the inner normal vector to $\Omega$, i.e.
$$u(y)= \int\limits_{\mathbb{R}^n_0} \kappa_n \frac{d(y,\mathbb{R}^n_0)}{d(x,y)^n}d\mu(x), y \in \mathbb{R}^n_+.$$
 Then for any sequence $\{x_i\}_{i=1}^{+\infty}$ in $\Omega$ non-tangentionally tending to $O$  the following asymptotic equality holds:
$$\tilde u(x_i)= u(x_i)(1+o(1))+o(1).$$
\end{theorem}
\begin{proof}[Proof of theorem \ref{sim1}]
 
 Applying (\ref{P3ii}) from theorem \ref{P3} we get
 \begin{align} \tilde u(y)= \int\limits_{\partial\Omega}P_L(x,y)d\mu(x)= \int \limits_{\partial\Omega} (\kappa_n \frac{d(y,\partial\Omega)}{d(x,y)^n}(1+o(1))+o(1)) d\mu(x)= \\
 \label{sim}= (1+o(1))\int \limits_{\partial\Omega} \kappa_n \frac{d(y,\partial\Omega)}{d(x,y)^n} d\mu(x) +o(1), y \to O. \end{align}
Analogously  $$u(y)= (1+o(1))\int \limits_{\partial\Omega} \kappa_n \frac{d(y,\partial\Omega)}{d(x,y)^n} d\mu(x) +o(1), y \to O.$$
 These two asymptotic equalities yield (i).

 If $u\geq 1$ in $\Omega$, then the last equality implies $\int \limits_{\partial\Omega} \kappa_n \frac{d(y,\partial\Omega)}{d(x,y)^n} d\mu(x)\gtrsim 1$. Hence $$ (1+o(1))\int \limits_{\partial\Omega} \kappa_n \frac{d(y,\partial\Omega)}{d(x,y)^n} d\mu(x) +o(1)=  (1+o(1))\int \limits_{\partial\Omega} \kappa_n \frac{d(y,\partial\Omega)}{d(x,y)^n} d\mu(x) .$$
 Thus $$ \tilde u(y) \sim \int \limits_{\partial\Omega} \kappa_n \frac{d(y,\partial\Omega)}{d(x,y)^n} d\mu(x) \sim u(y).$$

\end{proof}
%\begin{proof}[Proof of theorem \ref{sim2}]
 First, we prove the following lemma:
\begin{lemma} \label{dsim}
 For $z$ in $\Omega$ non-tangentionally  tending to $O$ and $s$ in $ \partial \Omega$ tending to $O$  $$d(z,s)\sim d(z,Pr(s)).$$
\end{lemma}
\begin{proof}[Proof of lemma \ref{dsim}]
 To prove lemma \ref{dsim} it is enough to show that $d(s,Pr(s))=o(d(z,Pr(s)))$. Let $f:\mathbb{R}^{n-1}\cap B_\varepsilon(O) \to R$ be the local parameterization of $\partial \Omega$ near $O$, so that $s=(\xi,f(\xi))$ for a $\xi \in \mathbb{R}^{n-1}$. Since $\partial \Omega$ is $C^1$-smooth, \begin{equation} 
\label{1} d(s,Pr(s))=|f(\xi)|=o(|\xi|).
\end{equation}  
 Represent $z$ as $(\eta, \tau)$, where $\eta  \in \mathbb{R}^{n-1}$ and $\tau \in \mathbb{R}$. Since $z$ tends to $O$ non-tangentionally then 
\begin{equation} \frac{|\eta|}{|\tau|}=O(1).
\end{equation}
Note that
\begin{align}
\label{2} o(|\xi|)=o(max(|\xi-\eta|,|\eta|))=o\left((\frac{|\eta|}{|\tau|}+1)max(|\xi-\eta|,|\tau|)\right)=\\
\label{3}=O(1)\cdot o(max(|\xi-\eta|,|\tau|))= o(d(z,Pr(s))).
\end{align}
 Combining formulas (\ref{1}), (\ref{2}), (\ref{3}) we obtain $d(s,Pr(s))=o(d(z,Pr(s)))$.
\end{proof}
 
\begin{proof}[Proof of theorem \ref{sim2}]
According to  (\ref{sim})
$$ \tilde u(y)=(1+o(1))\int \limits_{\partial\Omega} \kappa_n \frac{d(y,\partial\Omega)}{d(x,y)^n} d\tilde\mu(x) +o(1)=$$
for any $r>0$
$$=(1+o(1))\int \limits_{\partial\Omega\cap B_r(O)} \kappa_n \frac{d(y,\partial\Omega)}{d(x,y)^n} d\tilde\mu(x) +o(1) =$$
for sufficiently small $r$
$$= \left((1+o(1))\int \limits_{\partial\Omega\cap B_r(O)} \kappa_n \frac{d(y,\partial\Omega)}{d(x,y)^n} d\mu(Pr(x)) +o(1)\right)\mathop{\sim}\limits^{1+\delta}$$
according to lemma \ref{sim} for any $\delta>0$ $\exists r>0$:
$$\mathop{\sim}\limits^{1+\delta}\left((1+o(1))\int \limits_{\partial\Omega\cap B_r(O)} \kappa_n \frac{d(y,\partial\Omega)}{d(Pr(x),y)^n} d\mu(Pr(x)) +o(1)\right)=$$
$$=u(y)(1+o(1))+o(1). $$
 Thus for any $\delta>0$ we have
$\tilde u(y) \mathop{\sim}\limits^{1+\delta}\left( u(y)(1+o(1))+o(1)\right)$. 

\end{proof}
\subsection{Criterion of existence of non-tangentional limit}
\label{nontangential limit}
  We are going to formulate the theorem by W.Ramey and D.Ullrich. This theorem provides a criterion of existence of a non-tangentional limit at a fixed boundary point $P$ of the half-space for a positive harmonic function in terms of smoothness of its boundary measure  at P.  We need the notion of the strong derivative of measure \cite{WRDU}:
\begin{definition} \label{D1}
 Let $\Omega$ be a $C^1$-smooth domain and $\mu $ be a locally finite Borel measure on $\partial \Omega$. A sequence of balls $\{B_{r_i}(x_i)\}_{i=1}^{\infty}$ is called regular (with respect to $O\in \partial \Omega$) if the following properties hold:
\begin{enumerate}
\item $x_i \in \partial \Omega$ for any $i\in \mathbb{N}$ and $d(x_i,O)\mathop{\to}\limits_{i\to \infty}0$ .
\item $\exists K>0$: $\frac{1}{K} d(x_i,O)\leq r_i\leq Kd(x_i,O)$ for any $i\in \mathbb{N}$.
\end{enumerate}

 It is said that $\mu$ has a strong derivative at $O\in \Omega$ if for any  regular  sequence of balls $\{B_{r_i}(x_i)\}_{i=1}^{\infty}$ (with respect to $O$) there exist a finite limit
$$\lim\limits_{i\to \infty}\frac{\mu(B_{r_i}(x_i))}{S(B_{r_i}(x_i))},$$
where $S$ is the surface Lebesgue measure on $\partial \Omega$.
This limit is denoted by $D\mu(O)$.
\end{definition}
 If $\Omega$ coincides with $\mathbb{R}^n_+$ there is an equivalent definition in terms of weak convergence.
\begin{definition}
 Let $\mu$ be a locally finite measure on $\partial\mathbb{R}^n_+$. Define the family of measures $\{\mu_r\}_{r>0}$ by $$\mu_r(E)=\mu(rE)r^{-n+1}.$$ 
If there exist a number $A\geq0$, such that the family $\{\mu_r\}_{r>0}$ converge weakly to $A\cdot S$ as $r\to 0$, then $D\mu(O):=A$ is called the strong derivative of $\mu$ at origin $O$.

\end{definition}
 The following easy remarks are left without proof. Their sense  can be interpreted as follows: the property of a measure to have a derivative is stable under smooth transformations. 
\begin{remark}
Let $\Omega_1$ and $\Omega_2$ be $C^1$-smooth domains in $\mathbb{R}^n$. Suppose a map $T:\partial\Omega_1 \to \partial\Omega_2  $ is a $C^1$-smooth diffeomorphism. If a Borel measure $\mu$ on $\partial \Omega$ has a strong derivative at $x\in \partial\Omega_1$, then the measure $\tilde \mu$  defined by 
$$\mu (E)=\tilde \mu (T(E)).$$
on $\Omega_2$ also has a strong derivative at $T(x)$.
\end{remark}
\begin{remark} \label{pr}
 Let $\Omega$ be $C^1$-smooth domain and $\mu $ be a locally finite Borel measure on $\partial \Omega$. Suppose the hyperplane $\mathbb{R}^n_0$ is the tangent plane to $\partial \Omega$ at $O$. Consider $B_\varepsilon(O)$ where $\varepsilon$ is sufficiently small so that the orthogonal projection $Pr: \partial \Omega\cap B_\varepsilon(O) \to \mathbb{R}^n_0$ is injective.
Consider a finite Borel measure $\tilde \mu$ on $\mathbb{R}^n_0$ defined by 
$$\tilde\mu(E)=\mu(Pr^{-1}E).$$
Then $\mu$ has a strong derivative at $O$ and $D\mu(O)=A$ if and only if  $\tilde \mu$ has a strong derivative at $O$ and $D\tilde\mu(O)=A$.
\end{remark}
The theorem below is due to W.Ramey and D.Ullrich. It provides a criterion of existence of a non-tangentional limit at  a given boundary point $P$ of a half-space for a positive harmonic function in terms of smoothness of  its boundary measure at $P$.
\begin{theorem}[{\cite{WRDU}}] \label{NT1}
 Suppose $u$ is a positive harmonic function in $\mathbb{R}^n_+$ and $\mu$ is its boundary measure, $A \in [0,\infty)$. Then $u$ has a non-tangentional limit $A$ at $O\in\partial\Omega$ if and only if $\mu$ has a strong derivative at $O$ and $D\mu(O)=A$.
\end{theorem}
 It appears that the theorem above can be generalized as follows.
 \begin{theorem} \label{NT2}
 Let $\Omega$ be a $C^{2,\varepsilon}$-smooth bounded domain in $\mathbb{R}^n$, $n\geq 3$, $O \in \partial \Omega$, $A \in [0,\infty)$. Suppose $L\in L^+(\lambda,\alpha, \Omega)$, $u$ is a positive $L$-harmonic function in $\Omega$ and $\mu$ is its boundary measure. Then $u$ has a non-tangentional limit $A$ at $O$ if and only if its boundary measure has a strong derivative at $O$ and $D\mu(O)=A$.
\end{theorem}
\begin{proof}
 We will  deduce theorem \ref{NT2} from theorems \ref{NT1},\ref{sim2}.
 Without loss of generality we may assume that $O$ is the origin, the inner normals to $\mathbb{R}^n_+$ and $\partial \Omega$ at $O$ coincide  and the matrix $A(O)$ of the leading coefficients of $L$ at $O$ is the identity matrix. We can always get this  by a linear transfrom of the coordinates and a shift. 
Next, we choose an $\varepsilon>0$ so that the orthogonal projection $Pr: \partial \Omega\cap B_\varepsilon(O) \to \mathbb{R}^n_0$ is injective and define the Borel measure $\tilde\mu$ on $\mathbb{R}^n_0$ by 
$$\tilde\mu(E)= \mu(Pr^{-1}(E)).$$ 
Define $\tilde u$ as the harmonic continuation of $\tilde \mu$ into the halfspace $\mathbb{R}^n_+$.

 We are going to show that the following properties are equivalent:
\begin{enumerate}
\item $\mu$ has a strong derivative at $O$ and $D\mu(O)=A$,
\item $\tilde\mu$ has a strong derivative at $O$ and $D\tilde\mu(O)=A$,
\item $\tilde u$ has non-tangentional limit $A$ at $O$,
\item $u$ has non-tangentional limit $A$ at $O$.
\end{enumerate}
 Remark \ref{pr} implies 1 $\Leftrightarrow$ 2, and  2 $\Leftrightarrow$ 3 by theorem \ref{NT1}. 
 Theorem \ref{sim2} says that 
 $$u(y)=(1+o(1))\tilde u(y)+o(1) $$
 as $y$ tends non-tangentially to $O$.
 Hence 3 $\Leftrightarrow$ 4, and 1$\Leftrightarrow$4. 

\end{proof}
 The next theorem is due to Lynn Loomis ($n=2$) and Walter Rudin ($n\geq 2$). This theorem is very similar to theorem \ref{NT1} and provides a criterion of existence of a limit along the normal at a boundary point for a positive harmonic function. To formulate this theorem  we need the notion of the symmetric derivative of a measure.
\begin{definition}
  Suppose that a measure $\mu$ is concentrated on the boundary of a $C^{1}$-smooth domain $\Omega$. Let $S$ be the surface Lebesgue measure on $\partial \Omega$. We say that $\mu$ has a symmetric derivative $A$ at $O\in \partial \Omega$ if $\lim\limits_{r\to+0} \frac{\mu(B_r(O))}{S(B_r(O))}=A$ $(=: D_{sym}\mu(O))$.
\end{definition} 
\begin{theorem}[{\cite{LL}, \cite{WR}}] \label{NL1}
Suppose $u$ is a positive harmonic function in $\mathbb{R}^n_+$ and $\mu$ is its boundary measure. Then $u$ has a finite limit $A$ along the normal at $O\in\partial\Omega$ and this limit is $A$ if and only if $D_{sym}\mu(O)=A$.
\end{theorem}
 The next theorem extends theorems \ref{criterium} and \ref{NL1} to some class of elliptic operators and to sufficiently smooth domains.
\begin{theorem} \label{NL2}
Let $\Omega$ be a $C^{2,\varepsilon}$-smooth and bounded domain in $\mathbb{R}^n$, $n\geq 3$,  let $n(x)$ denote the unit normal (interior) vector at the point $x\in \partial \Omega$. Suppose $L\in L^+(\lambda,\alpha, \Omega)$, the matrix of the leading coefficients of $L$ at $O$ is the identity matrix, $\kappa \in (-1,n-1]$, and $A \in [0,+\infty)$. Let $u$ be a positive $L$-harmonic function in $\Omega$ and let $\mu$ be its boundary measure. Then  $u(x+n(x)t)t^{\kappa} \to A$ as $t \to +0$ if and only if $\frac{\mu({B_r(x)})}{r^{n-1}} r^{\kappa} \to C_\kappa A$ as $r\to+0$, where ${C_\kappa= \frac{\pi^{n/2}}{\Gamma(\frac{n-\kappa+1}{2})\Gamma(\frac{\kappa+1}{2})}}$.
\end{theorem}
 We omit the proof of this theorem here because it is parallel to the proof of theorem \ref{NT2}.
\subsection{Beurling minimal principle} \label{tBMP}
 The term "Beurling's minimal principle" was introduced in the paper \cite{VM}  where the result  of A.Beurling on the behavior of positve harmonic functions was extended to higher dimensions and generalized to positive solutions of elliptic operators in divergence form in sufficiently smooth domains. One of the ideas used in \cite{VM} concerned the asymptotic estimates for the Green function near the boundary and it allowed to go from a halfspace and the Laplace operator to $C^{1,\varepsilon}$ -smooth domains and elliptic operators in the divergence form. We will use this idea to show that the Beurling minimal principle holds for some class of elliptic operators in non-divergence form as well.
First, we are going to formulate the Beurling minimal principle for harmonic functions, it can be viewed as the condition on growth of a positive harmonic function along the sequence of points which ensures the boundary measure to have a point mass. 
 
\begin{definition}
  Suppose a sequence $\{z_i\}$ in $\Omega$ tends to $O\in \partial \Omega$ and is separated (i.e.
$\inf\limits_{i\neq j } \frac{d(x_i,x_j)}{d(x_i,\partial \Omega)}>0$).
	We call the sequence  $\{ z_i \}$ $L$-defining if for any positive $L$-harmonic function $u$ in 
$\Omega$ the inequalities $u(z_i)\geq \kappa P_L(O,z_i)$ imply the inequality $u(z)\geq \kappa P_L(z,O)$ for any $z\in\Omega$, in other words, the boundary measure of $u$ has a point mass at least $\kappa$ at  $O$. 
\end{definition}
  
\begin{theorem}[{\cite{AB},\cite{BD1},\cite{VM}}]
 Suppose  $\Omega$ is a $C^{1,\varepsilon}$-smooth and bounded domain in $\mathbb{R}^n$. Suppose that 
\begin{enumerate}[a)]
\item $L$ is the Laplace operator or
\item $L= div(a_{ij}\frac{\partial}{\partial x_i})$ is  a uniformly elliptic operator  in divergence form and the coefficients $a_{ij}$ are Holder continuous in $\overline \Omega$.
\end{enumerate}
 Then the separated sequence $\{ z_i \}$ tending to $O\in \partial \Omega$ is defining for $L$ if and only if 
 \begin{equation} \label{Beurling condition}
\sum\limits_{i} \left(\frac{d(z_i,\partial \Omega)}{d(z_i,O)} \right)^n= +\infty.
\end{equation}
\end{theorem}
 From now on, we assume that $\Omega$ is $C^{1,1}$-smooth. We say that the elliptic operator $L$ is good if 
\begin{enumerate}
\item There is one-to-one map from  positive $L$-harmonic functions onto finite Borel measures on $\partial \Omega$ provided by formula :
$$u(\cdot )= \int\limits_{\partial\Omega} P_L(x,\cdot)d\mu(x), $$  
where $\mu$ is a finite Borel measure  on $\partial \Omega$ and $u$ is positive $L$-harmonic in $\Omega$(as before,  we call $\mu$ the boundary measure of $u$).
\item There exists a positive constant $K$ such that the Poisson kernel of $L$ enjoys the following property:
\begin{equation} \label{PEE} \frac{1}{K}\frac{ d(y,\partial\Omega)}{ |x-y|^n}\leq P_L(x,y) \leq K \frac{ d(y,\partial\Omega)}{ |x-y|^n} \end{equation}
 for any $x\in \partial \Omega$ and $y\in \Omega$.
 In other words, the Poisson kernel for $L$ is comparable with the Poisson kernel for the Laplacian.

 For instance, any operator in $L^+(\lambda,\alpha,\Omega)$ is good if $\Omega$ is sufficiently smooth (see sections \ref{introduction},\ref{G-estimates}).
 In \cite{AILR} the Poisson kernel estimate (similar to (\ref{PEE})) for some class of elliptic operators of the form $div(A(x)\nabla_x)+ B(x)\nabla_x $ with singular drift terms is established. Similar estimates of the Poisson kernel for Schrodinger operators are obtained in \cite{ZZ} and \cite{AILR}.
\end{enumerate}
 The next theorem says that the Beurling minimal principle holds for any good elliptic operator (not necessarily in divergence form) in a $C^{1,1}$- smooth and bounded domain.
\begin{theorem} \label{BMP2}
Assume $\Omega$ is $C^{1,1}$-smooth and bounded domain in $\mathbb{R}^n$ and $L$ is a good elliptic operator in $\Omega$. The separated sequence $\{ z_i \}$ in $\Omega$ tending to $O\in \partial \Omega$ is defining for $L$ if and only if  $\{ z_i \}$ enjoys (\ref{Beurling condition}).
\end{theorem}
\begin{proof} 
It's sufficient to show that non-defining separated sequences for the Laplacian and $L$ coincide. Suppose that $\{ z_i \}$ is not defining for $L$.
Note that $\{ z_i \}$ is not $L$-defining if and only if there is a positive $L$-harmonic function $u$ in $\Omega$ such that its boundary measure $\mu$ has no point mass at $O$ and $u(z_i)\geq \kappa P_L(z_i,O)$ for a $\kappa>0$ and all $i$.
 Consider the harmonic continuation $\tilde u$ of $\mu$ into $\Omega$. Let us use the estimates of the $L$-Poisson kernel (\ref{PEE}) and the $\Delta$-Poisson kernel (\ref{pe}):
\begin{align*} \tilde u(y)= \int\limits_{\partial \Omega} P_\Delta(x,y)d\mu(x) \geq \frac{1}{K_1} \int\limits_{\partial \Omega} \frac{ d(y,\partial\Omega)}{ |x-y|^n} d\mu(x) \geq \\ \geq   \frac{1}{K_2K_1} \int\limits_{\partial \Omega} P_L(x,y)d\mu(x)=u(y), y\in \Omega. \end{align*}
 Hence $$\tilde u(z_i)\geq \frac{\kappa}{K_1K_2} u(z_i)\geq \frac{\kappa}{K_1K^2_2} \frac{ d(z_i,\partial\Omega)}{ d(z_i,O)^n} \geq \left(\frac{1}{K_1K_2}\right)^2\kappa P_\Delta(O,z_i).$$
 Thus $\{ z_i \}$ is not $\Delta$-defining. The converse implication 
\begin{equation*}  \{ z_i \} \hbox{ is not defining for } L  \Longleftarrow \{ z_i \} \hbox{ is not defining for } \Delta
\end{equation*}
 is obtained literally in the same way.
\end{proof}
 The next theorem is a straightforward consequence of theorem \ref{BMP2} and the asymptotic formula for harmonic measure (see theorem \ref{P3}).
\begin{theorem} \label{BMP3}
Assume $\Omega$ is a $C^{2,\varepsilon}$-smooth and bounded domain in $\mathbb{R}^n$. Suppose  $L \in\L(\lambda,\alpha,\Omega)$ and the matrix of the leading coefficients of $L$ at $O\in \partial\Omega$  is the identity matrix. Suppose that a separated sequence $\{ z_i \}$ tending to $O\in \partial \Omega$ satisfies (\ref{Beurling condition}):
Then for any positive $L$-harmonic function $u$ the asymptotic inequality 
$$\liminf \limits_{i \to \infty}\frac{u(z_i)}{\kappa \frac{d(z_i,\partial \Omega)}{d(z_i,O)^n}} \geq 1 $$
implies that its boundary measure has a point mass at least $\frac{\kappa}{\kappa_n}$ at $O$ and the asymptotic inequality
$$\liminf\limits_{z\in \Omega; z\to O} \frac{u(z)}{ \kappa \frac{d(z,\partial \Omega)}{d(z,O)^n}}\geq 1$$ 
 holds.
\end{theorem}

\section*{acknowledgements}

I am grateful to V.Havin for his guidance and criticism. I also thank I.Netuka and V.Mazya for helpful advice.
This research is supported by the Chebyshev Laboratory (Department of Mathematics
and Mechanics, St. Petersburg State University) under the RF Government grant
11.G34.31.0026, and by JSC "Gazprom Neft".

\end{document}